\documentclass[a4paper,oneside,11pt, reqno]{amsart}

\usepackage{mathrsfs}
\usepackage{amsfonts,amssymb,amsmath,amsthm}
\usepackage{url}
\usepackage{enumerate}

\usepackage{bbm}
\usepackage{float}
\usepackage{pstricks}
\usepackage{amscd}
\usepackage{amsmath}
\usepackage{amsxtra}
\usepackage[T1]{fontenc}
\usepackage[utf8]{inputenc}
\usepackage{lipsum}
\usepackage{tikz}
\usetikzlibrary{arrows}
\usetikzlibrary{calc}
\usepackage{hyperref}

\theoremstyle{plain}
\newtheorem{theorem}{Theorem}[section]

\newtheorem{proposition}[theorem]{Proposition}
\newtheorem{corollary}[theorem]{Corollary}

\theoremstyle{definition}

\newtheorem{example}[theorem]{Example}

\theoremstyle{remark}
\newtheorem{remark}[theorem]{Remark}

\newtheorem{case-new}{Case}

\numberwithin{equation}{section}

\usepackage{enumitem}
\usepackage{etoolbox}

\makeatletter
\@namedef{subjclassname@2020}{%
  \textup{2020} Mathematics Subject Classification}
\makeatother

\usepackage{mathrsfs}
\usepackage{epsfig}
\usepackage[active]{srcltx}

\newcommand{\ncom}{\newcommand}
\ncom{\bq}{\begin{equation}}
\ncom{\eq}{\end{equation}}
\ncom{\beqn}{\begin{eqnarray*}}
\ncom{\eeqn}{\end{eqnarray*}}
\ncom{\beq}{\begin{eqnarray}}
\ncom{\eeq}{\end{eqnarray}}
\ncom{\nno}{\nonumber}
\ncom{\rar}{\rightarrow}
\ncom{\Rar}{\Rightarrow}
\ncom{\noin}{\noindent}
\ncom{\bc}{\begin{centre}}
\ncom{\ec}{\end{centre}}
\ncom{\sz}{\scriptsize}
\ncom{\rf}{\ref}
\ncom{\sgm}{\sigma}
\ncom{\Sgm}{\Sigma}
\ncom{\dt}{\delta}
\ncom{\Dt}{Delta}
\ncom{\lmd}{\lambda}
\ncom{\Lmd}{\Lambda}
\ncom{\eps}{\epsilon}
\ncom{\pcc}{\stackrel{P}{>}}
\ncom{\dist}{{\rm\,dist}}
\ncom{\sspan}{{\rm\,span}}
\ncom{\im}{{\rm Im\,}}
\ncom{\sgn}{{\rm sgn\,}}
\ncom{\ba}{\begin{array}}
\ncom{\ea}{\end{array}}
\ncom{\eop}{\hfill{{\rule{2.5mm}{2.5mm}}}}
\ncom{\eoe}{\hfill{{\rule{1.5mm}{1.5mm}}}}
\ncom{\eof}{\hfill{{\rule{1.5mm}{1.5mm}}}}
\ncom{\hone}{\mbox{\hspace{1em}}}
\ncom{\htwo}{\mbox{\hspace{2em}}}
\ncom{\hthree}{\mbox{\hspace{3em}}}
\ncom{\hfour}{\mbox{\hspace{4em}}}
\ncom{\hsev}{\mbox{\hspace{7em}}}
\ncom{\vone}{\vskip 2ex}
\ncom{\vtwo}{\vskip 4ex}
\ncom{\vonee}{\vskip 1.5ex}
\ncom{\vthree}{\vskip 6ex}
\ncom{\vfour}{\vspace*{8ex}}
\ncom{\norm}{\|\;\;\|}
\ncom{\integ}[4]{\int_{#1}^{#2}\,{#3}\,d{#4}}
\ncom{\inp}[2]{\langle{#1},\,{#2} \rangle}
\ncom{\Inp}[2]{\Langle{#1},\,{#2} \Langle}
\ncom{\vspan}[1]{{{\rm\,span}\#1 \}}}
\ncom{\dm}[1]{\displaystyle {#1}}



\begin{document}
\title[Rank one perturbations of 2-isometries]{Rank one perturbations of 2-isometries}

\author[R. Nailwal]{Rajkamal Nailwal}
\address{Rajkamal Nailwal, Institute of Mathematics, Physics and Mechanics, Ljubljana, Slovenia.}
\email{raj1994nailwal@gmail.com \textup{(Corresponding author)}}

\begin{abstract} 
In this paper, we investigate when a rank-one perturbation of a $2$-isometry remains a $2$-isometry. As an application, we identify several classes of $2$-isometries on function spaces.
\end{abstract}
\maketitle

\section{Introduction}
Let $\mathcal{H}$ be a complex separable Hilbert space, and let $B(\mathcal{H})$ denote the unital $C^*$-algebra of bounded linear operators on $\mathcal{H}$. For any $T \in B(\mathcal{H})$, we denote the kernel and range of $T$ by $\operatorname{ker} T$ and $\operatorname{ran} T$, respectively. A closed subspace $U$ of $\mathcal{H}$ is said to be invariant under $T \in B(\mathcal{H})$ if $T(U)\subseteq U.$

An operator $T \in B(\mathcal{H})$ is called a \emph{$2$-isometry} if
\[
\Delta_{2}(T) := I - 2T^*T + T^{*2}T^2 = 0.
\]
It follows immediately from the definition that $T$ is a $2$-isometry if and only if
\[
\|x\|^2 - 2\|Tx\|^2 + \|T^2x\|^2 = 0, \quad \forall\, x \in \mathcal{H}.
\]

Let $u, v \in \mathcal{H}$, and define the bounded linear operator $K$ on $\mathcal{H}$ by
\[
K := u \otimes v, \quad \text{where } (u \otimes v)(x) = \langle x, v \rangle u, \quad  x \in \mathcal{H}.
\]
It is easy to observe that if both $u$ and $v$ are nonzero, then $\operatorname{ran} K$ is one-dimensional. In this case, $K$ is referred to as a \emph{rank-one operator} on $\mathcal{H}$.

Given $T\in B(\mathcal{H}),$ a rank $1$ perturbation of $T$ is defined by
\beq\label{perT}
\widetilde{T} :=  T + u \otimes v
\eeq
for some nonzero $u,v.$

In this paper, we investigate the following:

\textit{Under what conditions does the rank one perturbation of $2$-isometry (as defined by \eqref{perT})  remain a $2$-isometry?} 

Rank one perturbation has always been an important topic in operator theory. This is particularly useful in understanding the continuity of spectra, invariant subspaces, and the structural changes in operators (e.g. \cite{CGP23,DS22,DH21,T22}). Note that isometries are two isometries. Our investigation is motivated by \cite{YN1986}, \cite{ST2007}, which discuss rank one and compact perturbations of isometries, respectively.

Operators that are $2$-isometries form an important class of operators and have attracted considerable attention in recent research (see \cite{ACJS2019,MMS2016, RS1991} and references therein). Our interest in $2$-isometries comes from function theory. A classical example is the multiplication operator by $z$ on the Dirichlet space of the unit disc, which is a well-known $2$-isometry that is not an isometry. S. Richter investigated the class of $2$-isometries in connection with the structural theory of  $z$-invariant subspaces in the Dirichlet space (see \cite{R1988}). As $2$-isometries are generalizations of isometries, they serve as a motivating example class for conjectures and properties that may not hold for arbitrary bounded operators but might still extend beyond isometries (e.g. \cite{AO2004}).

The following theorem is the main result of this paper:

\begin{theorem}\label{main}
Let $T \in B(\mathcal{H})$ be a $2$-isometry and let $u, v \in \mathcal{H}$ be nonzero vectors with $\|v\| = 1$.  
Consider the rank-one operator $K = u \otimes v$, and define  
\[
S := \{ x \in \ker K : Tx \in \ker K \}.
\]  
Then the perturbed operator $\widetilde{T} := T + K$ is a $2$-isometry if and only if $v \in \ker \Delta_{2}(\widetilde{T})$ and exactly one of the following conditions holds:
\begin{enumerate}
    \item $S = \ker K$, i.e. $\ker K$ is an invariant subspace of $T$.
    
    \item The following two conditions are satisfied:
    \begin{enumerate}
        \item[$(a)$] 
        $
        \Delta_{2}(\widetilde{T})(S) \subseteq S
        $ (equivalently, $\Delta_2(\widetilde T)(S^\perp\cap\ker K)\subseteq S^\perp\cap\ker K); $
        \item[$(b)$] $
        \|u\|^2 = -2 \big(\gamma + \operatorname{Re} \langle u, T v \rangle \big)
        $
        where $\gamma \in \mathbb{R}$ is  given by
        \[
       \gamma= \operatorname{Re} \left( \frac{\langle T^* P_{\ker K} T^* u, x \rangle}{\langle T^* v, x \rangle} \right)
        \]
        where $x \in (S^\perp \cap \ker K) \setminus \{0\}.$
    \end{enumerate}
\end{enumerate}
Here, $P_{\ker K}$ denotes the orthogonal projection onto $\ker K$, and $\operatorname{Re}(z)$ denotes the real part of a complex number $z$.
\end{theorem}

\begin{remark}
\begin{enumerate}
    \item If \((i)\) does not hold, then \(S^\perp \cap \ker K\) is one-dimensional (see Remark~\ref{st-cap}), hence \(\gamma\) is unique. 
    Moreover, the condition \emph{(ii)(a)} can be checked easily.
    \item Using the property \(u \otimes (\alpha v) = (\overline{\alpha}\, u) \otimes v\), one may always choose \(v\) to have norm \(1\).
\end{enumerate}
\end{remark}

We devote the following section to the proof of Theorem~\ref{main}. The core idea behind the proof is to construct appropriate orthogonal decompositions of $\mathcal{H}$ that simplify the computation of norms.

\section*{Proof of Theorem~\ref{main}}
For the reader’s convenience, we recall a few elementary properties of rank-one operators.

\begin{proposition}
Let $u, v \in \mathcal{H}$. Then the following identities hold for all $x, y \in \mathcal{H}$ and  $V \in B(\mathcal{H})$:
\begin{enumerate}
    \item $(u \otimes v)(x \otimes y) = \langle x, v \rangle\, u \otimes y$;
    \item $V(u \otimes v) = (Vu) \otimes v$;
    \item $(u \otimes v)V = u \otimes (V^*v)$.
\end{enumerate}
\end{proposition}

\begin{proof}
 This is elementary to verify.
\end{proof}
The following simple fact describes the set $S$.
\begin{proposition}
    \label{lem:S_equiv_invariance}
Let $T\in B(\mathcal{H})$ and $u,v\in\mathcal{H}\setminus{\{0\}}$. Set $K:=u\otimes v$ and
\[
S:=\{x\in\ker K:\ Tx\in\ker K\},\qquad \widetilde T:=T+K.
\]
Then
$
S=\big(\operatorname{span}\{v,T^*v\}\big)^\perp.
$
\end{proposition}

\begin{proof}
Since $K=u\otimes v$, we have $\ker K=\{x:\langle x,v\rangle=0\}=v^\perp$. Moreover,
\[
Tx\in\ker K\ \Longleftrightarrow\ \langle Tx,v\rangle=0\ \Longleftrightarrow\ \langle x,T^*v\rangle=0,
\]
so $S=v^\perp\cap(T^*v)^\perp=(\operatorname{span}\{v,T^*v\})^\perp$.
\end{proof}
\begin{remark} \label{st-cap}
    We now record that since $S^\perp = \operatorname{span}\{v, T^* v\}$ and 
$\ker K = v^\perp$ we have
\[
S^\perp \cap \ker K =
\begin{cases}
\operatorname{span}\!\left\{T^* v - \langle T^* v, v\rangle v \right\}, & \text{if } T^* v \not\parallel v, \\[4pt]
\{0\}, & \text{if } T^* v \parallel v.
\end{cases}
\]
That is, $S^\perp \cap \ker K$ coincides with the orthogonal projection of $T^* v$ onto $v^\perp$.

\end{remark}

We are now ready to prove Theorem~\ref{main}.

\begin{proof}[Proof of Theorem~\ref{main}]
Let $K = u \otimes v$ be the rank one operator with non-zero vectors $u,v \in \mathcal{H},$ and $\|v\|=1.$ We will need the following orthogonal decompositions in the upcoming steps:
\beq\label{orth-dec}
\mathcal{H} = \ker K \oplus (\ker K)^\perp=\ker K^{*} \oplus (\ker K^{*})^\perp.
\eeq
Note that $\operatorname{ker} K = [\operatorname{span}(v)]^\perp,\operatorname{ker} K^{*} = [\operatorname{span}(u)]^\perp$.\\ 
\textbf{Necessity part:} Assume that $\widetilde{T} := T + K$ is a $2$-isometry. Immediately, $v \in \operatorname{ker} \Delta_{2}(\widetilde{T})$ and $(ii)(a)$ is satisfied.
Let $x \in \operatorname{ker} K=[\operatorname{span}(v)]^\perp$. 
Then 
\[
\widetilde{T}x = Tx \quad \text{and} \quad \widetilde{T}^2x = T^2x + KTx.
\]
Since $\widetilde{T}$ is a $2$-isometry, we have
\begin{equation} \label{eq:2iso-x}
\|x\|^2 - 2\|Tx\|^2 + \|T^2x + KTx\|^2 = 0.
\end{equation}
Expanding the last term
\begin{align*}
\|T^2x + KTx\|^2 &= \|T^2x\|^2 + 2\operatorname{Re} \langle KTx, T^2x \rangle + \|KTx\|^2.
\end{align*}
Plugging this into \eqref{eq:2iso-x}, and using the $2$-isometry condition for $T$, we obtain
\beq\label{KT-eq}
2\operatorname{Re} \langle KTx, T^2x \rangle + \|KTx\|^2 = 0.
\eeq
Let $S=\{x \in \operatorname{ker} K: Tx\in \operatorname{ker}K\}.$ Clearly, $S$ is a closed subspace of $\operatorname{ker} K.$ Note that if $S= \operatorname{ker} K,$ then we are done i.e. $[\operatorname{span}(v)]^{\perp}$ is an invariant subspace for $T.$ 

Let if possible there exists a nonzero $x \in S^{\perp}\cap \operatorname{ker} K.$ By \eqref{orth-dec}, we have $T(x)=u_1+u_2,$ where $u_1 \in \operatorname{ker} K, u_2 \in (\operatorname{ker}K)^{\perp}$  and $u_2\neq 0.$ By \eqref{KT-eq}, we have
\beqn
2\operatorname{Re}(\langle u_2, K^{*}Tu_1\rangle+\langle Ku_2, Tu_2  \rangle) + \|Ku_2\|^2 = 0.
\eeqn
Again we decompose $T(u_1)=v_1+v_2,$  where $v_1 \in \operatorname{ker} K^*,$ and $v_2 \in (\operatorname{ker}K^*)^{\perp}.$ We obtain

\[
2\operatorname{Re}(\langle u_2, K^{*}v_2\rangle+\langle Ku_2, Tu_2  \rangle) + \|Ku_2\|^2 = 0.
\]
Equivalently 
\[
2\operatorname{Re}(\langle(u\otimes v) u_2, v_2\rangle+\langle (u\otimes v)u_2, Tu_2  \rangle) + \|(u \otimes v) u_2\|^2 = 0.
\]
By \eqref{orth-dec}, we can take $u_2=\alpha v, \alpha \in \mathbb C\setminus \{0\}.$ and $v_2=\beta u, \beta \in \mathbb C.$ Then we have
 \[
2\operatorname{Re}(\langle(u\otimes v)\alpha v, \beta u\rangle+ \langle (u\otimes v)\alpha v, \alpha T v  \rangle) + \|(u \otimes v)\alpha v\|^2 = 0.
\]
By definition of $u \otimes v,$ and the condition $\|v\|=1,$ we get
 \[
2\operatorname{Re}(\alpha \overline{\beta}\|u\|^2+|\alpha|^2\langle u,  T v  \rangle) + |\alpha|^2\|u\|^2 = 0.
\]
Dividing  by $|\alpha|^2,$ we get
 \[
2\operatorname{Re}\left(\frac{\overline{\beta}}{\overline{\alpha}} \|u\|^2+\langle u,  T v  \rangle\right) + \|u\|^2 = 0.
\]
Note that $\alpha =\langle Tx, v \rangle, \beta =\frac{\langle Tu_1, u \rangle}{\|u\|^2}.$ Thus we have

 \beqn
     2\operatorname{Re}\left(\frac{\langle u,Tu_1 \rangle}{{\langle v, Tx \rangle}}+\langle u,  T v  \rangle\right) + \|u\|^2 = 0.
 \eeqn
 This yields that 
 \begin{equation*}
     2\operatorname{Re}\left(\frac{\langle u,Tu_1 \rangle}{{\langle v, Tx \rangle}}\right)= -\|u\|^2-2\operatorname{Re}\langle u,  T v  \rangle .
 \end{equation*}
Let $\gamma=\frac{-\|u\|^2-2\operatorname{Re}\langle u,  T v  \rangle}{2}$, we have
$$\operatorname{Re}\left(\frac{\langle u, Tu_1\rangle}{\langle v, Tx\rangle}\right)=\gamma, \quad  x \in (S^{\perp}\cap \operatorname{ker} K) \setminus \{0\}.$$
Let $P_{\operatorname{ker} K}$ be the orthogonal projection onto $\operatorname{ker} K.$ Since $u_1= P_{\operatorname{ker} K}Tx,$ we have

$$ \operatorname{Re}\frac{\langle u, TP_{\operatorname{ker} K}Tx\rangle}{\langle v, Tx\rangle}=\gamma, \quad  x \in (S^{\perp}\cap \operatorname{ker} K) \setminus \{0\},$$
Equivalently 
\beqn 
\operatorname{Re}\frac{\langle T^*P_{\operatorname{ker} K}T^*u, x\rangle}{\langle T^*v, x\rangle}=\gamma, \quad  x \in (S^{\perp}\cap \operatorname{ker} K) \setminus \{0\}.
\eeqn
This completes the proof of the necessity part.\\

\noindent\textbf{Sufficiency:}
Assume that $v$ satisfies the $2$-isometry condition for $\widetilde T:=T+K$, i.e., $\Delta_2(\widetilde T)v=0$, and that either \emph{(i)} $S=\ker K$ or \\
\emph{(ii)} The following are satisfied: \begin{enumerate}
    \item[$(a)$] $\Delta_2(\widetilde T)(S)\subseteq S,$   
    \item[$(b)$] $
        \|u\|^2 = -2 \big(\gamma + \operatorname{Re} \langle u, T v \rangle \big)
        $
        where $\gamma \in \mathbb{R}$ is  given by
        \[
       \gamma= \operatorname{Re} \left( \frac{\langle T^* P_{\ker K} T^* u, x \rangle}{\langle T^* v, x \rangle} \right)
        \]
        where $x \in (S^\perp \cap \ker K) \setminus \{0\}.$
\end{enumerate}
 Since $\Delta_2(\widetilde T)$ is self-adjoint and $\Delta_2(\widetilde T)v=0$, the one-dimensional subspace $[\operatorname{span} v]$ reduces $\Delta_2(\widetilde T)$; hence its orthogonal complement $\ker K=v^\perp$ is invariant for $\Delta_2(\widetilde T)$. As
\[
\mathcal H=[\operatorname{span} v]\oplus \ker K,\qquad 
\ker K=S\oplus (S^\perp\cap\ker K),
\]
it suffices to prove that $\Delta_2(\widetilde T)$ vanishes on $S$ and on $S^\perp\cap\ker K$.

\emph{Case A: $x\in S$.}
Then $Kx=0$ and $KTx=0$, so
\[
\langle \Delta_2(\widetilde T)x,x\rangle
=\|x\|^2-2\|Tx\|^2+\|T^2x\|^2
=\langle \Delta_2(T)x,x\rangle
=0,
\]
because $T$ is a $2$-isometry. Under (ii)(a) (or trivially under (i) when $S=\ker K$), the subspace $S$ is invariant for $\Delta_2(\widetilde T)$; by the polarization identity it follows that $\Delta_2(\widetilde T)\!\upharpoonright_S=0$.

\emph{Case B: $x\in S^\perp\cap\ker K$, $x\neq 0$.}
Write $Tx=u_1+\alpha v$ with $u_1\in\ker K$ and $\alpha=\langle Tx,v\rangle\neq 0$ (since $x\notin S$). Then
\[
KTx=\alpha u,\qquad \widetilde T x=Tx,\qquad \widetilde T^2 x=T^2x+\alpha u.
\]
Hence
\begin{align*}
\langle \Delta_2(\widetilde T)x,x\rangle
&=\|x\|^2-2\|Tx\|^2+\|T^2x+\alpha u\|^2 \\
&=\langle \Delta_2(T)x,x\rangle
+2\,\operatorname{Re}\langle \alpha u, T^2x\rangle+\|\alpha u\|^2 \\
&=|\alpha|^2\!\left(2\,\operatorname{Re}\!\Big[\frac{\overline{\beta}}{\overline{\alpha}}\|u\|^2+\langle u,Tv\rangle\Big]+\|u\|^2\right),
\end{align*}
where $\beta:=\dfrac{\langle Tu_1,u\rangle}{\|u\|^2}$. By $(ii)(b)$,
\[
\operatorname{Re}\!\left(\frac{\overline{\beta}}{\overline{\alpha}}\|u\|^2\right)=\gamma
\quad\text{and}\quad
\|u\|^2=-2\big(\gamma+\operatorname{Re}\langle u,Tv\rangle\big),
\]
which yields $\langle \Delta_2(\widetilde T)x,x\rangle=0$. Since $\Delta_2(\widetilde T)$ leaves $S^\perp\cap\ker K$ invariant (self-adjointness and invariance of $S$ and $\ker K$), polarization gives $$\Delta_2(\widetilde T)\!\upharpoonright_{S^\perp\cap\ker K}=0.$$
Combining the three orthogonal summands $[\operatorname{span} v]\oplus S\oplus(S^\perp\cap\ker K)=\mathcal H$, we conclude that $\Delta_2(\widetilde T)=0$ on $\mathcal H$. Hence $\widetilde T$ is a $2$-isometry.
\end{proof}
The following is an immediate consequence of Theorem~\ref{main}.
\begin{corollary}
   Let $u,v \in \mathcal{H}$ be nonzero vectors with $\|v\|=1,$ and $I$ be the identity operator on $\mathcal{H}.$ Then $I+u \otimes v $ is a $2$-isometry if and only if $v$ satisfy the $2$ isometry condition for  $I+u \otimes v.$ 
\end{corollary}
We now present an example of an isometry whose rank-one perturbation is also an isometry and it satisfies Theorem~\ref{main}(ii).
\begin{example}
Let $H = \mathbb{C}^2$, and define the isometry
\[
V = \begin{pmatrix} 0 & 1 \\ 1 & 0 \end{pmatrix},
\]
which is in fact unitary. Let $ e_1 = (1,0)^{\top},e_2= (0,1)^{\top}$. From \cite{YN1986}, it follows that $\widetilde{V}:=V-2e_1\otimes e_2$ is an isometry. In particular, it is $2$-isometry. It is not difficult to verify that it satisfies Theorem\ref{main}(ii) with $S$ being empty and  $\gamma=0$ 
\end{example}

\allowdisplaybreaks
\section{Applications}
In this section, we study a rank $1$ perturbation of $M_z$ on the classical Dirichlet space and the Hardy space on the unit bidisk. We will just briefly explain the Dirichlet space.

The \textit{Dirichlet space} \( \mathcal{D} \) is the space of holomorphic functions on the open unit disk \( \mathbb{D} \subset \mathbb{C} \), defined by:
\[
\mathcal{D} = \left\{ f(z) = \sum_{k=0}^{\infty} a_k z^k \ \middle| \ \|f\|_{\mathcal{D}}^2 := \sum_{k=1}^{\infty} (k+1)|a_k|^2 < \infty \right\}.
\]
The following is the well-known inner product on $\mathcal{D}$ which makes it a Hilbert space: for  $f(z) = \sum_{k=0}^{\infty} a_k z^k,g(z) = \sum_{k=0}^{\infty} b_k z^k,$ define
      $$
     \langle f,g\rangle_{\mathcal{D}}:= \sum_{k=1}^{\infty} (k+1)a_k \overline{b_k}.
    $$
 Multiplication by $z$, i.e. $M_z$, is a bounded operator on this space and also a $2$-isometry.
For a detailed exposition, we refer the reader to \cite{FKMR2014}. 

The following theorem identify a class of $2$-isometries of the type $M_z+p(z)\otimes1$ where $p(0)=0.$ Such operators are analytic i.e. $$\cap_{n=0}^{\infty}(M_z+p(z)\otimes1)^n(\mathcal{H})=\{0\}$$
(see \cite[Theorem~1.2]{CGP23}).
\begin{theorem}\label{p-per}
Let $p(z) = \sum_{i=1}^{k} a_i z^i$, where $a_i \in \mathbb{C}$. Then the operator $M_z + p(z) \otimes 1$ is a $2$-isometry on $(\mathcal{D}, \|\cdot\|_{\mathcal{D}})$ if and only if 
\[
 \sum_{i=1}^k i|a_i|^2=-2\operatorname{Re}(a_1).
\]
\end{theorem}

\begin{proof}
Since $[\operatorname{span} (1)]^{\perp}$ is an invariant subspace for $M_z,$  it suffices to verify that $1$ satisfy the $2$-isometry condition for  $\widetilde{M_z} := M_z + p(z) \otimes 1.$

We compute:
\[
\widetilde{M_z}(1) = z + p(z), \quad \widetilde{M_z}^2(1) = z^2 + z p(z).
\]
Then,
\[
\|\widetilde{M_z}(1)\|^2 = \sum_{i=1}^{k}(i+1)|b_i|^2, \quad \|\widetilde{M_z}^2(1)\|^2 = \sum_{i=2}^{k+1}(i+1)|c_i|^2,
\]
where
\[
b_i =
\begin{cases}
a_i, & i \neq 1, \\
a_1 + 1, & i = 1,
\end{cases}
\quad \text{and} \quad
c_i =
\begin{cases}
a_{i-1}, & i \neq 2, \\
a_1 + 1, & i = 2.
\end{cases}
\]

Now consider the $2$-isometry identity:
\begin{align*}
\|1\|^2 - 2\|\widetilde{M_z}(1)\|^2 + \|\widetilde{M_z}^2(1)\|^2 
&= 1 - 2\sum_{i=1}^k (i+1)|b_i|^2 + \sum_{i=2}^{k+1} (i+1)|c_i|^2 \\
&= 1 - 4|a_1 + 1|^2 - 2\sum_{\substack{i=2}}^k (i+1)|a_i|^2 \\
&\quad + \sum_{\substack{i=3}}^{k+1} (i+1)|a_{i-1}|^2 + 3|a_1 + 1|^2 \\
&= 1 - |a_1 + 1|^2 - 2\sum_{i=2}^k (i+1)|a_i|^2 \\
&\quad+ \sum_{i=2}^{k} (i+2)|a_i|^2 \\
&= -2\operatorname{Re}(a_1) - \sum_{i=1}^k i|a_i|^2.
\end{align*}
This completes the proof.
\end{proof}
Next we see some special cases where $p(z)=\alpha z^n.$
\begin{corollary}
Let $n \geq 0, \alpha \neq 0$. Then the operator $M_z + \alpha z^n \otimes 1$ is a $2$-isometry on $(\mathcal{D},\|\cdot\|_{\mathcal{D}})$ if and only if 
 $n = 1$ and $|\alpha|^2 = -2\operatorname{Re}(\alpha).$
\end{corollary}
\begin{proof} $n\geq1$ follows from Theorem~\ref{p-per} by applying it to the polynomial $p(z) = \alpha z^n$. Now we will show that for $n=0, \alpha \neq 0,$   then $M_z + \alpha \otimes 1$ is never a  $2$-isometry. Note that
\[
\widetilde{M_z}(1) = z + \alpha, \quad \widetilde{M_z}^2(1) = z^2 + \alpha z + \alpha^2.
\]
Evaluating the 2-isometry condition on the constant function $f(z)=1,$
\[
\|1\|^2 - 2\|\widetilde{M_z}(1)\|^2 + \|\widetilde{M_z}^2(1)\|^2 = 1 - 2(|\alpha| + 2) + (|\alpha|^2 + 2|\alpha| + 3) = |\alpha|^2.
\]
Since $\alpha \neq 0,$  $\widetilde{M_z}$ is not a $2$-isometry.
\end{proof}

To state the next result, we need the following.  The \textit{Hardy space on the unit bidisc} is defined as follows:

$$
H^2(\mathbb{D}^2) = \left\{ f(z) = \sum_{m, n = 0}^{\infty} a_{m,n} \, z_1^m z_2^n \  \middle| \ \  \|f\|_{H^2}^2 := \sum_{m, n = 0}^{\infty} |a_{m,n}|^2 < \infty \right\}.
$$

It is easy to verify that the multiplication operators $M_{z_i}$, for $i=1,2$, are isometries on $H^2(\mathbb{D}^2)$. In particular, these operators are $2$-isometries. 

We note that, so far, we have not utilized Theorem \ref{main}(ii) to identify $2$-isometry operators. While this condition is indeed subtle, we have succeeded in identifying an operator that satisfies it. We now present this result and emphasize that the approach outlined here can be used to obtain more classes of $2$-isometries.

\begin{corollary}\label{cond-ii}
The operator $\widetilde{M} := M_{z_1} + (-z_1^2 + z_2) \otimes z_1$ is a $2$-isometry on $(H^2(\mathbb{D}^2),\|\cdot\|_{H^2(\mathbb{D}^2)}).$
\end{corollary}
\begin{proof}
In view of Theorem \ref{main}, it suffices to verify that $z_1\in \operatorname{ker}\Delta_2(\widetilde{M}),$ and Theorem \ref{main}(ii)(a,b) holds. Note that $\widetilde{M} z_1 = z_2$ and $\widetilde{M}^* z_2 = z_1$ and
\[
\widetilde{M}^* \widetilde{M} z_1 = z_1, \quad 
\widetilde{M}^2 z_1 = z_1 z_2, \quad 
\widetilde{M}^{*2} \widetilde{M}^2 z_1 = z_1.
\]
Hence
\[
\big(I - 2\widetilde{M}^* \widetilde{M} + \widetilde{M}^{*2} \widetilde{M}^2\big) z_1
= 0.
\]
 Thus we have $z_1\in \operatorname{ker}\Delta_2(\widetilde{M}).$ Now
since $S = \{1, z_1\}^\perp$ (using Remark~\ref{st-cap}), $ S^\perp\cap \operatorname{ker} K = \operatorname{span}\{1\}.$ Note that $\Delta_2(\widetilde{M})1=0.$ 
Therefore, condition \emph{(ii)(a)} is satisfied.

 Since $M_{z_1}^*(-z_1^2 + z_2) = -z_1$, we obtain
$$
M_{z_1}^* P_{[\operatorname{span}(z_1)]^{\perp}} M_{z_1}^{*}(-z_1^2 + z_2) = 0.
$$
We also have
$$
\| -z_1^2 + z_2 \|^2 = -2 \operatorname{Re} \langle -z_1^2 + z_2, M_{z_1} z_1 \rangle.
$$
Thus, the conditions of Theorem \ref{main}(ii) are satisfied. Hence, the proof is complete.
\end{proof}

\noindent{\bf Author's contribution:} The paper is written by Rajkamal Nailwal solely.\\
\medskip
\noindent{\bf Availability of data and materials:}\\
\medskip
No data was used for the research described in the article.\\
\medskip
\noindent{\bf Funding:} Not applicable\\
\medskip
\noindent{\bf Declaration of competing interest:}\\
\medskip
The author declares that he has no competing interests.
  
{}

\end{document}